\documentclass[12
pt,twoside]{amsart}
\usepackage{amssymb}
\usepackage{a4wide}
\usepackage[latin1]{inputenc}
\usepackage[T1]{fontenc}
\usepackage{times}

\def\hpq0{h^{p,q}_{\leq 0}}
\def\Hpq0{\H_{\leq 0}^{p,q}}

\def\dbar{\bar\partial}
\def\ddbar{\partial\dbar}

\def\C{{\mathbb C}}

\def\X{{\mathcal X}}
\def\D{D_{GM}}

\def\dof{\dot{\varphi_t}}

\def\K{{\mathcal K}}

\def\H{{\mathcal H}}

\def\u{{\mathbf u}}
\def\v{{\mathbf v}}

\def\L{{\mathcal L}}

\def\Re{{\rm Re\,  }}

\def\L{{\mathcal L}}

\def\be{\begin{equation}}
\def\ee{\end{equation}}

\newtheorem{thm}{Theorem}[section]
\newtheorem{lma}[thm]{Lemma}

\newtheorem{prop}[thm]{Proposition}

\theoremstyle{definition}

\theoremstyle{remark}

\newtheorem{preremark}{Remark}
\newtheorem{preex}{Example}

\numberwithin{equation}{section}

\begin{document}

\title[]
{Strict and non strict positivity of direct image bundles.}

\address{B Berndtsson :Department of Mathematics\\Chalmers University
  of Technology \\
  and Department of Mathematics\\University of G\"oteborg\\S-412 96 G\"OTEBORG\\SWEDEN,\\}

\email{ bob@math.chalmers.se}

\author[]{ Bo Berndtsson}

\begin{abstract}
{This paper is a sequel to \cite{Berndtsson}. In that paper we studied
  the vector bundle associated to the direct image of the relative
  canonical bundle of a smooth K\"ahler morphism, twisted with a
  semipositive line bundle. We proved that the curvature of a such
  vector bundles is always semipositive (in the sense of Nakano). Here
  we adress the question if the curvature is strictly positive when
  the Kodaira-Spencer class does not vanish. We prove that this is so
  provided the twisting line bundle is stricty positive along fibers,
  but not in general. } 
\end{abstract}

\bigskip

\maketitle

\section{Introduction}

Let $p:\X\rightarrow Y$ be a smooth proper holomorphic fibration of
complex manifolds of relative dimension $n$, and let $\L\rightarrow
\X$ be a holomorphic line bundle 
equipped with a smooth metric of semipositive curvature. The direct
image sheaf of the relative canonical bundle twisted with $\L$, 
$$
p_*(\L+ K_{\X/Y}),
$$
is then associated to a vector bundle, $E$, over $Y$ with
fibers
$$
E_y=H^0(\X_y, K_{\X_y}+\L|_{\X_y}),
$$
see e g \cite{Berndtsson}. In \cite{Berndtsson} we have shown that if
$\X$ is Kähler, the natural $L^2$-metric on $E$ has nonnegative
curvature in the sense of Nakano. In this paper we shall, in two
special cases,  discuss more
explicit formula for the curvature, which enables us to determine when
the curvature is strictly positive. We will all the time consider only
the case of a one dimensional base, but the computations generalize
 to the case of a base of higher dimension. For the moment however,
 these higher dimensional computations seem to give little more with
 regard to Nakanopositivity than what is already contained in
 \cite{Berndtsson}, so we omit them here . (The curvature in the sense
 of Griffiths can be 
 obtained from 
 the case of one dimensional base.)

We concentrate on two particular cases that are somewhat opposite. The
first case is when $\L$ is trivial, so that we are dealing with the
direct image of the relative canonical bundle itself. The second case
is when $\L$ is strictly positive on all fibers.

In the first  case  the
semipositivity theorem is already contained in the work of Griffiths,
see \cite{Griffiths}, p 34, and further developed by Fujita, \cite{Fujita}. 

When the fibration is trivial it is clear that the curvature 
vanishes, so one expects that the positivity of $E$ will depend on how
far from being trivial the fibration is. This is measured by the
Kodaira-Spencer class. We recall its definition in the next section;
for the moment it is enough to remember that, at a point $t$ in the
base,  it is given by an element
in $H^{0,1}(\X_t, T^{1, 0}(\X_t))$, which vanishes if the fibration is
trivial (to first order) at $t$. Let us denote this class by $\K_t$;
it is represented by  $\dbar$-closed $(0,1)$-forms with values in the
holomorphic tangent bundle of the fiber. 

An element in $E_t$ is a holomorphic $(n,0)$-form, $u$, on $X_t$. The
Kodaira-Spencer class acts on $u$ in a natural way: If $k_t$ is a
vectorvalued  $(0,1)$ form in $\K_t$, which locally decomposes as
$$
k_t=w\otimes v,
$$
where $w$ is scalar valued and $v$ is a vector field, then first we let
the vector field part of $k_t$ act on $u$ by contraction
$$
k_t.u:= w\wedge\delta_v u.
$$
This gives a globally defined $\dbar$-closed form of bidegree $(n-1,1)$ and
$$
\K_t.u:= [k_t.u],
$$
an element in $H^{n-1, 1}(X_t)$. The following theorem is due to
Griffiths, see \cite{Griffiths} and further references there, but we
shall also discuss how it follows from the formalism in
\cite{Berndtsson} in sections 3 and 4 of this paper. 
\begin{thm}
Let $\Theta^E$ be the curvature of $E$ with the natural
$L^2$-metric. Then
\be
\langle \Theta^E u, u\rangle=\|\K_t.u\|^2.
\ee
The right hand side is the norm of the class $\K_t.u$ with respect to
the given Kähler metric, i e the norm of its unique harmonic
representative. It does not depend on the choice of Kähler metric.
\end{thm}

It is clear from Theorem 1.1 that if the Kodaira-Spencer class vanishes,
then $\Theta^E u=0$ for any $u$ in $E_t$. It may however very well
happen that $\K_t.u$, and hence $\Theta^E u$, vanish for some choice
of $u$ even if $\K_t\neq 0$. This happens precisely when the class $\K_t$
contains a current, not cohomologous to zero, supported on the zero
divisor of $u$. In section 3 
we will give  explicit examples of this, when the fibers
of $p$ are Riemann surfaces of genus at least 2. We shall see that
any compact Riemann surface of genus at least 2 can be put as  the central fiber
of some fibration as above, in such a way that the curvature is
degenerate (i e not strictly positive), although the Kodaira-Spencer
class is not zero. We will also give examples where
the Kodaira-Spencer class is not zero at the central fiber, but the
curvature $\Theta^E$ is zero applied to any $u$ in $E_0$. This is more
exceptional; when fibers are compact Riemann surfaces this can be done
precisely when the central fiber is hyperelliptic of genus at least
three.

The other particular case that we study is when $\L$ is nontrivial,
and has a metric $\phi$ of strictly positive curvature along the
fibers of $\X$; we will call this the relatively ample case. Of course
we do not assume that the curvature is strictly positive on the total
space $\X$, but just on the fibers. We will see that in this case the
moral is opposite to when $\L$ is trivial: If the Kodaira-Spencer class
is not zero, then the curvature $\Theta^E$ is strictly positive. We
emphazise again that this holds as soon as the metric on $\L$ is
positive fiberwise, and no assumption is made on positivity of $\L$ in
'horizontal' directions. Thus the main conclusion of this paper is
that 
 degeneracy of
$\Theta^E$ implies that the fibration is (infinitesimally) trivial,
provided $\L$ is relatively ample,
whereas this does not hold in general.

In the relatively ample case we have already
discussed strict positivity in \cite{Berndtsson} and
\cite{2Berndtsson} for the case when the {\it fibration} is
trivial. We shall now extend this to nontrivial fibrations.

To formulate our result in that case we first recall a few well known
facts. On $\X$ there is a certain smooth vector field, $V_\phi$
depending on the metric, see \cite{Siu},\cite{Schumacher},\cite{2Berndtsson} and
\cite{Berman}. It is defined as follows. Choose local coordinates
$(t,z)$ on $\X$ that respect the fibration so that $p(t,z)=t$. Let the
metric $\phi$ be represented by a local function $\varphi$, with
respect to some local trivialization of $\L$ over the coordinate
neighbourhood. Put
$$
\dof=\frac{\partial\varphi}{\partial t},
$$
and define on each fiber $X_t$ a vector field $W_\varphi$ by
$$
\delta_{W_\varphi}i(\ddbar)_z\varphi=\dbar_z\dof.
$$ 
This defines $W_\varphi$ uniquely if $i\ddbar\varphi>0$ on fibers;
$W_\varphi$ is the {\it complex gradient} of $\dof$. There is no reason
why $W_\varphi$ should be independent of the choices we made, but as
proved in \cite{Schumacher}, the field
$$
V_\phi:= \frac{\partial}{\partial t} - W_\varphi
$$
is a well defined smooth field on $\X$ if we let $t$ vary (we will
reprove this in Lemma 4.1) . Moreover,
$\dbar_z V_\phi$, $\dbar V_\phi$ restricted to a fiber $X_t$, is a
representative of 
the Kodaira-Spencer class $\K_t$. We denote this representative
$k^\phi_t$; it vanishes precisely when $V_\phi$ is holomorphic on $X_t$. 

We need one more ingredient in order to state our next result. Let
\be
c(\phi):= \frac{\partial\dof}{\partial\bar t} -|\dbar_z\dof|^2,
\ee
where we measure $\dbar\dof$ with respect to the metric $\omega^t
=i\ddbar \varphi|_{X_t}$. As proved in \cite{Semmes}
\be
(i\ddbar\phi)_{n+1}=c(\phi)(i\ddbar_z\phi)_n\wedge dt\wedge d\bar t.
\ee
Hence $c(\phi)$ is a globally well defined function on $\X$ which
measures the (lack of) strict positivity of $i\ddbar\phi$ on $\X$. 

\begin{thm}
 Assume $\phi$ is a metric on $\L$ with $i\ddbar\phi$ strictly positive
 on fibers. The
 curvature of the $L^2$-metric on $E$ is then given at $t$  by
\be
\langle\Theta^E u,u\rangle=\int_{X_t} c(\phi) |u|^2 e^{-\phi} +\langle
  (\Box' +1)^{-1}\eta,\eta\rangle,
\ee
with $\eta=k^\phi_t.u$. Here $\Box' =\nabla'(\nabla')^*+(\nabla')^*\nabla'$ is the
Laplacian on $\L|_{X_t}$-valued forms on $X_t$ defined by the
$(1,0)$-part of the Chern connection on  $\L|_{X_t}$.
\end{thm}

Several conclusions can be drawn from this. First we note that if
$\Theta^E u=0$ for some $u$ in $E_t$, then $c(\phi)=0$ and
$k^\phi_t.u=0$. The second condition implies that $k^\phi_t=0$. (This
is because $k^\phi_t$ is now one fixed smooth form, and not a cohomology
class, so if it vanishes ouside the zero divisor of $u$, it vanishes
identically.) Hence the Kodaira-Spencer class
vanishes, which as we saw is not necessarily the case for $\L$
trivial. It also follows that $\Theta^E u=0$ for {\it all} $u$ in
$E_t$.  Moreover
$V_\phi$ is holomorphic on $X_t$ and we shall  see in the last section
that even 
$$
\frac{\partial}{\partial \bar t}V_\phi
$$ 
vanishes at $t$, so $V_\phi$ is holomorphic to first order also in 
 directions transverse to the fiber. Finally, $V_\phi$ can be lifted
 to a field on   $\L|_{X_t}$ with the same properties, and the
 infinitesimal automorphism of $\L$ defined by the lift preserves the
 metric on $\L$. All in all, if the curvature is degenerate in the
 relatively ample case, the fibers and the line bundle over it just
 move by an infinitesimal automorphism as we vary the point in the base. 

In the special case when we have a fibration where the fibers are
Riemann surfaces of genus at least 2, $\L=K_{\X/\Delta}$ and $\phi$
restricts to the Kähler-Einstein metric on the fibers $X_t$, the metric
on $E$ is dual to the  Weil-Petersson metric. Explicit formulas for the
curvature of $E$ were in this case found by Wolpert, see \cite{Wolpert},
and also Liu-Sun-Yau, \cite{Yau et al} for much more on this area. In
section 5 we show how Wolpert's formula follows from (1.4), via a
theorem of Schumacher.

One might also note  that since $\Box'\geq 0$, (1.4) also gives an
estimate from above of the curvature
\be
\langle\Theta^E u,u\rangle\leq\int_{X_t} \left (c(\phi) + |\dbar
  V_\phi|^2\right )|u|^2 e^{-\phi}.
\ee

The plan of the rest of the paper is as follows. In the next section we
summarize some background material, including the curvature formula
from \cite{Berndtsson} that we will use. Although brief, the
discussion here is essentially self contained, and somewhat easier
than in \cite{Berndtsson}, since we only deal with the case of
onedimensional base. In section 3 we prove Theorem
1 and give the  examples when the curvature is degenerate (or even 0)
even though the 
Kodaira-Spencer class does not vanish.(The reader who is only
interested in the examples can go directly to section 3.1.)  In
section 3.2 we discuss the 
relation between the Chern connection on $E$ (for $\L=0$) and the
Gauss-Manin connection on the Hodge bundle of the fibration. This
gives an alternate proof of Theorem 1.1 and also proves Griffiths
theorem that $E$ is a holomorphic subbundle of the Hodge bundle. The
last section contains the proof of Theorem 1.2 and the consequences of
it mentioned above. 

Finally I would like to sincerely thank the referee for a very detailed reading
of the manuscript and many suggestions for improvement.Thanks also go
to Mihai Paun and Per Salberger for helping to find the Max Noether
theorem, used in section 3.1.

\section{Background material}

\subsection{ The bundle $E$ and its metric}

We consider the general setting from the introduction and recall the
setup from \cite{Berndtsson}. In particular, we assume throughout
that $\X$ is K\"ahler, and we let $\omega$ denote some choice of
K\"ahler form on $\X$.  Since the discussion is local we take the base to
be $\Delta$, the unit disk in $\C$.  Let us first note that it follows from the
Ohsawa-Takegoshi extension theorem that $E$ is a holomorphic vector
bundle in this 
case. A section of 
the bundle $E$ is  a function that maps $t$ in $\Delta$ to an
element of $E_t$, i e to a global holomorphic $(n,0)$-form, $u_t$ with
values in $\L$ on $X_t$. (From now we will denote the fibers $X_t$
instead of $\X_t$.) $E$ has a holomorphic structure such that
$u_t$ is a holomorphic section if and only if the $(n+1,0)$-form,
defined fiberwise by
$u_t\wedge dt$, is a holomorphic section of $K_\X +\L$. The
computations in  \cite{Berndtsson} are based on the notion of a {\it
  representative } of $u=u_t$. We say that $\u$ is a representative of
$u$ if $\u$ is an $(n,0)$-form on $\X$, with values in $\L$ such that
$\u$ {\it restricts} to $u_t$ on fibers $X_t$. This means that 
$$
i_t^*(\u)=u_t
$$
where $i_t$ is the natural inclusion map from $X_t$ to
$\X$. Representatives are not unique; any two differ by a term
$$
dt\wedge v
$$
where $v$ is an $(n-1,0)$-form. 

If $\phi$ is a metric on $\L$ we get a natural $L^2$-metric on $E$ by
$$
\|u_t\|^2_t=\int_{X_t} |u_t|^2 e^{-\phi}.
$$
The expression $|u|^2e^{-\phi}$ should here be interpreted as
$$
c_n u\wedge\bar u e^{-\phi};
$$
it is a well defined volume form and can be integrated over fibers.  
Here and in the sequel, $c_n=i^{n^2}$ is a unimodular constant chosen
so that we get a 
positive form.

The bundle $E$ is thus a hermitian holomorphic vector bundle and as
such has a Chern connection
$$
D= D'
+D''.
$$
In terms of a representative of the section, the connection can be
described as follows. First (locally), since $\dbar\u=0$ on fibers,
$$
\dbar\u=d\bar t\wedge \nu +dt\wedge\eta.
$$
( We use here a different sign convention from the one in
\cite{Berndtsson}.)
Here $\nu$ is of bidegree $(n,0)$ and $\eta$ is of bidegree
$(n-1,1)$. The forms $\nu$ and $\eta$ are not uniquely determined but
their restrictions to fibers are, and since 
$$
 dt\wedge d\bar t\wedge\dbar\nu =dt\wedge\dbar^2\u =0
$$
$\nu$ is holomorphic on fibers. Then
$$
D''u=\nu d\bar t.
$$
 In particular,
$u$ is a holomorphic section of $E$ if and only if, $\dbar\u\wedge
dt=0$, i e if and only if
$$
\dbar\u=dt\wedge\eta
$$
for some $\eta$. Changing representative to $\u -dt\wedge v$ changes
$\eta$ to $\eta+\dbar v$. Thus the cohomology class of $\eta$ in
$H^{n-1,1}(X_t,\L)$ is well defined, and any element in this
cohomology class can be obtained from some representative of the
section. We shall see in Lemma 2.2 that this class is $-\K_t.u$. 

Given a representative, we can express the norm of a section by
$$
\|u_t\|^2_t= p_*(c_n \u\wedge\bar\u e^{-\phi}),
$$
the push forward of an $(n,n)$-form on $\X$. Using this, and the
definition of Chern connection, one finds that
$$
(D' u)_t=P(\mu) dt,
$$
where 
$$
\partial^\phi \u= e^\phi\partial e^{-\phi}\u=dt\wedge\mu
$$
and $P(\mu)$ is the orthogonal projection of the $(n,0)$-form $\mu$ on
the space of holomorphic $(n,0)$-forms. 

As in \cite{Berndtsson} we can now compute the Laplacian of 
$$
\|u_t\|^2_t
$$
with respect to $t$ by computing $i\ddbar p_*(c_n \u\wedge\bar\u
e^{-\phi})$. This uses that $\partial$ and $\dbar$ commute with the
pushforward operator, and the result is
$$
i\ddbar p_*(c_n \u\wedge\bar\u e^{-\phi})=
-p_*(c_n i\ddbar\phi\wedge\u\wedge\bar\u e^{-\phi}) +
c_n\int_{X_t}\eta\wedge\bar\eta e^{-\phi} dV_t+\|\mu\|_t^2dV_t,
$$
if $u_t$ is a holomorphic section of $E$.

Under this assumption  we moreover have the standard formula 
$$
\Delta\|u_t\|^2_t=-\langle \Theta^E u_t,u_t\rangle +\|D'u_t\|^2_t.
$$
Combining the two we get  for the curvature of
$E$ ($dV_t=i 
dt\wedge d\bar t$)
\be
\langle \Theta^E u_t,u_t\rangle_t= \|D'u_t\|^2_t -\|\mu\|^2_t +p_*(c_n
i\ddbar\phi\wedge\u\wedge\bar\u e^{-\phi})/dV_t
-c_n\int_{X_t}\eta\wedge\bar\eta e^{-\phi}= 
\ee
$$
=-\|P_{\perp}\mu\|^2_t +
p_*(c_n i\ddbar\phi\wedge\u\wedge\bar\u e^{-\phi})/dV_t
-c_n\int_{X_t}\eta\wedge\bar\eta e^{-\phi},
$$
where $P_{\perp}$ is the orthogonal projection on the orthogonal
complement of holomorphic forms. Notice that every term in the right
hand side here depends  on the choice of representative, whereas the
left hand side does not. The following lemma from \cite{Berndtsson}
tells us that we can 
choose representative so that the right hand side becomes manifestly
nonnegative.
\begin{lma} Let $u$ be a holomorphic section of $E$, and $t$ a point
  in the base $\Delta$. Then there is a representative $\u$ such that
  $\mu$ is holomorphic on $X_t$ and $\eta$ is primitive on $X_t$ (i e
  $\eta\wedge \omega=0$ on $X_t$, where $\omega$ is the Kähler form on $\X$ ).
\end{lma}
\begin{proof} We will sketch the proof here since the proof in
  \cite{Berndtsson} is carried out in the general case of a higher
  dimensional base and therefore is more complicated. Take $t=0$

Let  $\u$ be an arbitrary representative and recall that
$\dbar\u=dt\wedge \eta$ and $\partial^{\phi}\u=dt\wedge \mu$. Changing
to a different representative $\u+dt\wedge v$, where $v$ is a
$(n-1,0)$-form changes $\eta$ and $\mu$ to $\eta-\dbar v$ and
$\mu-\partial^{\phi}v$ respectively. We want to choose $v$ in such a
way that
\be
\omega\wedge(\eta-\dbar v)=0
\ee
on $X_0$, and 
\be
\mu-\partial^\phi v
\ee
is holomorphic on $X_0$. Let $\alpha=v\wedge \omega$ on $X_0$, so that
$\alpha$ is an $(n,1)$-form on $X_0$. Then the first equation becomes
$$
\eta\wedge\omega=\dbar\alpha.
$$
Let us first see that this equation is solvable, or in other words
that the cohomology class of $\eta$ is primitive: Since
$\u\wedge\omega$ is of bidegree $(n+1,1)$ we can write
$$
\u\wedge\omega=dt\wedge u'
$$
where $u'$ is a well defined $(n-1,1)$-form on fibers. Applying the
$\dbar$-operator on $\X$ we get
$$
dt\wedge\eta\wedge\omega=\dbar\u\wedge\omega=-dt\wedge\dbar u'.
$$
Hence $\eta\wedge\omega=-\dbar u'$ on fibers, so $\eta\wedge\omega$ is
$\dbar$-exact on all fibers. Hence $(2,2)$ is solvable.

The second equation, (2.3), is satisfied if $\dbar^*\alpha=\mu_\perp$,
where $\mu_\perp$ is the projection of $\mu$ on the orthogonal
complement of holomorphic $(n,0)$-forms. Since this space is precisely
the range of $\dbar^*$, (2.3) is solvable as well. 

Let $\alpha_1$ solve $\dbar\alpha_1=\eta\wedge\omega$ on $X_0$ and
take $\alpha_1$ to be orthogonal to the kernel of $\dbar$ . Then
$\alpha_1$ is orthogonal to the range of $\dbar$ so
$\dbar^*\alpha_1=0$. Let $\alpha_2$ solve $\dbar^*\alpha_2=\mu_\perp$
and take $\alpha_2$ orthogonal to the kernel of $\dbar^*$. Then
$\alpha_2$ is orthogonal to the range of $\dbar^*$ so
$\dbar\alpha_2=0$. Thus $\alpha=\alpha_1+\alpha_2$ solves both
equations, and we are done.

\end{proof}

\bigskip

\noindent Since 
$$
 -c_n\int_{X_t}\eta\wedge\bar\eta e^{-\phi}=\|\eta\|^2
$$
for $\eta$ primitive it follows that
$$
\langle \Theta^E u_t,u_t\rangle_t=p_*(c_n
i\ddbar\phi\wedge\u\wedge\bar\u e^{-\phi})/dV_t+\|\eta\|^2 \geq 0.
$$
This choice of representative will be used again in the proof of
Theorem 1.1 in section 3. The proof of Theorem 1.2 however, is based
on a different choice of representative, using the extra structure
provided by the line bundle $\L$.
\subsection{The Kodaira-Spencer class}
We first recall the definition of the
K-S-class in a suitable form. Let $V$ be a smooth vector field of type
$(1,0)$ on $\X$ which maps to the field $\partial/\partial t$ on
$\Delta$ under the derivative of the map $p$ from $\X$ to
$\Delta$. Then $\dbar V$ is a $(0,1)$ form on $\X$ with values in the
bundle of tangent vectors tangential to the fiber. Its restriction to a
fiber $X_t$ ( $i_t^*(\dbar V)$), $\kappa_t$ is a $\dbar$-closed form
with values in $T^{1,0}$  of the fiber. The
cohomology class it defines is the Kodaira-Spencer class, $\K_t$ at $t$. 

If $u_t$ is an element in $E_t$, we can let the Kodaira-Spencer class
operate on $u_t$ to obtain $\K_t.u_t$. This class is defined as the
cohomology class of $\kappa_t.u_t$, where we let the vectorvalued
form $\kappa_t$ operate on $u_t$ by contraction as described in the
introduction. This is then a 
cohomology class in $H^{n-1,1}(X_t, \L|_{X_t})$. Notice that when
$n=1$ and (say) $\L$ is trivial, $\kappa_t$ is a $(0,1)$-form with
values in $-K_{X_t}$, and if $u$ is a holomorphic 1-form on $X_t$ then
$\kappa_t. u$ is just the product, defining a scalar $(0,1)$-form. 
\begin{lma} Let $\u$ be any representative of a holomorphic section
  $u$ and let $\eta$ be defined by $\dbar\u=dt\wedge\eta$ as
  above. Then
$$
\eta +\kappa_t.u =\dbar(\delta_V \u)
$$
on fibers. In particular, $\eta$ and $-\kappa_t.u$ define
the same class in  $H^{n-1,1}(X_t, \L|_{X_t})$.
\end{lma}

\begin{proof}
On $\X$ we have
$$
\dbar(\delta_V \u)=\delta_{\dbar V}\u
+\delta_V(dt\wedge\eta)=\delta_{\dbar V}\u+\eta
-dt\wedge\delta_V\eta.
$$
When we restrict to a fiber $X_t$, the last term disappears and we get
$$
\dbar(\delta_V u)=\kappa_t.u+\eta.
$$
\end{proof}

\subsection{The Gauss-Manin connection} The content of this subsection
will only be used in section 3.2 and is not necessary to understand the
proofs of theorems 1.1 and 1.2.

The fibration $p$ of $\X$ over $\Delta$ is smoothly (locally) trivial, so
there is a map $F=(p, f)$ from $\X$ to $\Delta\times X_0$ which is
 fiber preserving, diffeomorphic and equals the identity map on the
 central fiber. If $F'$ is another such map, it is related to $F$ by
$$
F'=G\circ F
$$
where $G$ is a fiberpreserving map from $\Delta\times X_0$ to
itself. Hence $G$ is given as $G(t,x)=(t, G_t(x))$ where $G_t$ is a
smooth family of diffeomorphisms of $X_0$, which are moreover
homotopic to $G_0= $Id. 

\bigskip

Let $H$ be the Hodge bundle over $\Delta$, i e the vector bundle whose
fiber over a point $t$ is
$$
H_t= H^n(X_t, \C),
$$
the $n$th de Rham cohomology of $X_t$. Then $H$ is a trivial vector
bundle with a trivialization given by
$$
e_j(t)= i^*_t\circ f^*(e_j^0),
$$
$e_j^0$ being some arbitrary basis of $H_0$. If we replace $F=(p,f)$
by $F'=(p,f')$ as above,
$$
i_t^*\circ(f')^*=i_t^*\circ f^*\circ G_t^*=i_t^*\circ f^*,
$$
where the last equality follows since $G_t$ homotopic to the identity
map implies that $G_t^*=id$ on the cohomology. Our trivialization is
therefore independent of the choice of $F$.

We define the {\it
  Gauss-Manin} connection on $H$, by
$$
\D e_j=0;
$$
this is clearly independent of the choice of basis $e_j$. Clearly
$\D^2=0$, so the Gauss-Manin connection is flat. 

\bigskip

Next there is  a quadratic hermitian form on each $H_t$ by
$$
\langle v,v\rangle_t=c_n\int_{X_t} v\wedge\bar v.
$$
This form is indefinite, but non degenerate (and well defined) on the
cohomology groups $H_t$. If $v_t$ is a smooth section of $H$ we can as
before represent it by an $n$-form on the total space, $\v$ such that
$\v$ restricts to (an element of the cohomology class) $v_t$ on
$X_t$. By construction, a section 
satisfying $\D v=0$ can be represented by a form of the type
$$
\v= f^*(v_0)
$$
where $v_0$ is a closed form on $X_0$, so that $\v$ is closed on
$\X$. Moreover, for any smooth section,
$$
\langle v, v\rangle = c_np_*(\v\wedge\bar \v).
$$

\bigskip

It follows that if $\D v=0$ and $\D u=0$, then 
$$
d\langle v, u\rangle=0,
$$
so more generally
$$
d\langle v, u\rangle=\langle \D v, u\rangle + \langle v, \D u\rangle.
$$
In other words, the Gauss-Manin connection is compatible with our
hermitian form, so it is the Chern connection for the complex
structure $\D''$ and the hermitian form $\langle, \rangle_\cdot$.

\section{The untwisted case  }

In this section we discuss the case when $\L$ is trivial so that 
$\phi=0$. Then  the curvature formula (2.1) becomes

\be
\langle \Theta^E u_t,u_t\rangle_t= -\|P_{\perp}\mu\|^2_t
 -c_n\int_{X_t}\eta\wedge\bar\eta
\ee

\noindent We decompose $\eta =\eta_p +\eta_{\perp}$, where $\eta_p$ is
primitive on $X_t$ and $\eta_\perp$ is orthogonal to the space of primitive
forms (all with respect to a Kähler metric on $\X$ restricted to
$X_t$). Then 
$$
-c_n\int_{X_t}\eta\wedge\bar\eta= \|\eta_p|^2
-\|\eta_\perp\|^2=\|\eta\|^2 -2\|\eta_\perp\|^2.
$$
Inserting this in (3.1) we find
\be
\langle \Theta^E u_t,u_t\rangle_t= -\|P_{\perp}\mu\|^2_t+\|\eta\|_t^2
-2\|\eta_\perp\|_t^2\leq \|\eta\|_t^2. 
\ee
This holds for any choice of representative $\u$. When we let the
representative vary, the corresponding $\eta$ ranges through an
entire cohomology class. By Lemma 2.2
that cohomology class is $-\K_t.u$. By Lemma 2.1 we can get equality in
(3.2) by choosing $\u$ so that $\mu$ is holomorphic and $\eta $ is
primitive. This choice must thus give us the representative of $-\K_t.u$
of minimum norm, i e the harmonic representative of the class. Hence
$$
\langle \Theta^E u_t,u_t\rangle_t=\|\K_t.u\|^2
$$
so Theorem 1.1 is proved.

The conclusion of this is that the curvature of $E$ degenerates in a
certain direction $u_t$ over a point $t$ if the Kodaira-Spencer class
vanishes after multiplication by $u_t$. This may well happen even if
$\K_t$ itself is nonzero, as we see in the next section. 

\bigskip
\subsection{ Examples of degenerate curvature.}

The construction of the examples is a direct consequence of  the
following (well known) lemmas. 

\begin{lma} Let $X$ be a compact Riemann surface of genus at least 2
and let $\mu$ be a class in
$H^{0,1}(X,T^{1,0}(X))=H^{0,1}(X,-K_X)$. Then there is a smooth proper
fibration $\X\rightarrow \Delta$ over a disk $\Delta$ such that
$\X_0=X$ and the Kodaira-Spencer class $\K_0=\mu$.
\end{lma}

\begin{proof} This is a consequence of basic properties of
  Teichm\"uller space. Teichm\"uller space $T_g$ is a complex manifold
  of dimension $3g-3$ consisting of equivalence classes of Riemann
  surfaces of genus $g$.  Over $T_g$ there is a smooth proper fibration
  $\mathcal C_g\rightarrow T_g$ such that the fiber over a point $t$ in
  $T_g$ is a compact Riemann surface in the class $t$, and all compact
  Riemann
  surfaces of genus $g$ appear as fibers over some point(s) . (This is a
  fundamental result of Earle and Eells, \cite{Earle-Eells}.) The tangent
  space to $T_g$ at $t$ is (isomorphic to) 
$H^{0,1}(\mathcal C_g,  -K_{(\mathcal C_g)_t})$. We need only take a disk in
$T_g$, centered at a point representing $X$ with tangent vector at the
origin equal to $\mu$, and let $\X$ be $\mathcal C_g$ restricted to
that disk.
\end{proof}

\begin{lma}Let $X$ be a compact Riemann surface of genus at least
  2. Then for any $u$ in $H^0(X,K_X)$ there is a nonzero class $\mu$ in $
  H^{0,1}(X,-K_X)$ such that $\mu u=0$ in $H^{0,1}(X)$.
\end{lma}
\begin{proof} This follows from a simple count of dimensions. Given
  $u$ multiplication by $u$ defines a linear map from 
$H^{0,1}(X,-K_X)$ to $H^{0,1}(X)$. Since the dimension of
$H^{0,1}(X,-K_X)$ is $3g-3$ by the Riemann-Roch theorem, while the
dimension of $H^{0,1}(X)$ is $g$, the map cannot be injective if $g$ is
greater than 1. 

This can also be seen more concretely in the following way. (We denote
by $\mu$ also a smooth form representing the cohomology class $\mu$.) That $\mu
u$ vanishes in cohomology means that
$$
\mu u=\dbar v
$$
for some smooth function $v$ . This implies
$$
\dbar \frac{v}{u} =\mu +v \dbar\frac{1}{u}.
$$
Hence $\mu$ is cohomologous to the $\K_X$-valued current $-v \dbar
1/u$, supported in the zero set of $u$. To find a $\mu$ which is not
zero in cohomology, but is annihilated by multiplication by $u$ it
suffices conversely to find a $\K_X$-valued current with measure
coefficients, $\nu$,  supported in the 
zeroset of $u$, which is not exact. For this it is enough to take
$\nu$ supported in one single point. If $\nu=\dbar w$ then $w$ is a
meromorphic section of $-K_X$ with exactly one pole. But such sections
do not exist since the number of zeros minus the number of poles of
any section of $-K_X$ equals the degree of $-K_X$ which is $2-2g<-1$.
\end{proof}

To get curvature that is not only degenerate but vanishes completely,
we need a final  lemma.
\begin{lma}
Let $X$ be a compact Riemann surface. Then  there is a nonzero class $\mu$ in $
  H^{0,1}(X,-K_X)$ such that $\mu u=0$ in $H^{0,1}(X)$ for any $u$ in
  $H^0(X,K_X)$ if and only if $X$ is hyperelliptic of genus at least 3.
\end{lma}
\begin{proof} That $\mu u=0$ in   $H^{0,1}(X)$ means that the integral
 $$
\int_X (\mu u)\wedge v=0
$$
for any $v$ in   $H^0(X,K_X)$. This means that in the duality pairing
between $ H^{0,1}(X,-K_X)$ and $H^0(X,2K_X)$, $\mu $ is annihilated by
any section of $2K_X$ that factors as a product of two sections of
$K_X$, and also of course by any linear combination of such
sections. On the other hand, $\mu=0$ means that $\mu$ is annihilated
by any section of $2K_X$. The question is therefore if any section of
$2K_X$ is a linear combination of products of sections of $K_X$. By a
theorem of M Noether this is the case if $X$ is not hyperelliptic of
genus greater than 2 ( see \cite{Farkas-Kra} p 149), but not the case
if $X$ is hyperelliptic of genus at least 3 (see \cite{Farkas-Kra} p
98).  
\end{proof}

\begin{prop}

1. Let $X$ be a compact Riemann surface of genus at least
  2. Then there is a smooth proper fibration $\X$ over a disk $\Delta$
  , having
  $X$ as central fiber, such that its Kodaira-Spencer class does not
  vanish at 0, but the curvature of the direct image bundle of
  $K_{\X/\Delta}$ is degenerate (i e not strictly positive) at 0.

2. Assume $X$ is hyperelliptic of genus at least 3. Then there is a
smooth proper fibration $\X$ over a disk $\Delta$ 
  , having
  $X$ as central fiber, such that its Kodaira-Spencer class at 0 does not
  vanish, but the curvature of the direct image bundle of
  $K_{\X/\Delta}$ is zero at 0.
\end{prop} 

\begin{proof} 

1. By Lemma 3.2 there is a non zero class $\mu$ in
  $H^{0,1}(X, -K_X)$ that is annihilated by multiplication by a
  canonical form $u$ on $X$. By Lemma 3.1 $\mu$ is the Kodaira-Spencer
  class of some fibration, and by Theorem 1.1, $\Theta^E u=0$. 

2. Follows from Lemma 3.3 in the same way.
\end{proof}

{\bf Remark:} A reader more knowledgeable than the author can probably
compare this section with the discussion of the local period mapping
for Riemann surfaces in \cite{2Griffiths}, p 842-843.
\qed
\subsection{The relation to the Gauss-Manin connection}

A holomorphic $(n,0)$-form on a fiber defines a unique element in $H^n$, so
our  bundle $E$ is a smooth subbundle of $H$. By a theorem of
Griffiths  (that we will reprove below)  $E$ is in fact a holomorphic
subbundle (\cite{Voisin},p 250), if we give $H$ the 
holomorphic structure induced by the Gauss-Manin connection (so that
$\dbar$ on section of $H$ is the $(0,1)$-part of $\D$).

If we have a smooth section of $E$ and $V$ is a vector field on
$\Delta$ we may thus take $D_Vu$, where $D$ is the connection defined
in section 2,  at a point $t$ and obtain a
holomorphic $(n,0)$-form on $X_t$, or we can apply $(\D)_V u$ and obtain a
cohomology class in $H^n(X_t,\C)$. The relation between the two is
given by the next theorem.
\begin{thm}
Let $u$ be a smooth section of $E$. Then
\be
\D u = [D u]-[\K_t.u]dt.
\ee
\end{thm}

This statement should be read as follows: The class $[\K_t.u]$ lies in
$H^{n-1,1}$, which by Hodge theory can be thought of as a subspace of
$H^n$. So the right hand side of (4.1) should be read as 'the
$H^n$-class defined by $D u$ minus the $H^n$-class defined by
$[\K_t.u]$, multiplied by $dt$'.

 \begin{proof}Let $v$ be a section of $H$ with $\D
v=0$. We represent $v$ by $\v$, a closed form on $\X$. Then, if $u$ is
a smooth section of $E$, represented by $\u$,
$$
\langle \D u, v\rangle =d\langle u,v\rangle =d p_*(c_n \u\wedge\bar\v)=c_n
p_*(d\u\wedge \bar\v)=c_n 
p_*((\mu+\eta)\wedge \bar\v)dt +p_*(\nu,\bar\v)d\bar t.
$$
Since $\nu$ and $\mu+\eta$ are closed, and the form $\langle, \rangle
$ is nondegenerate on the cohomology groups, we see that $\D'u=[\mu+\eta]
dt$ and $\D''u=\nu d\bar
t=D''u $ . This formula holds for any choice of
representative. By Lemma 2.1, we can always choose our representative
in such a way 
that $\mu$ is holomorphic, and then $\mu dt=D' u$ . By Lemma 1.3,
$[\eta]=-[\K_t.u]$.
This completes the proof of Theorem 1.3. 
\end{proof}
\bigskip

\noindent Formula (4.1) contains many properties of the direct image
bundle $E$. First, we note that, with the convention described above,
$\D''= D''$, i e the $(0,1)$ part of the two connections agree. This
is what lies behind Griffiths theorem, \cite{Voisin} p 250  that $E$ is
a {\it holomorphic} 
subbundle of $H$. Since $\D''^2=0$, the equation implies that
$(D'')^2=0$, so $D''$ 
defines an integrable complex structure on $E$,
\cite{Donaldson-Kronheimer}. Thus $E$ has a (local) 
frame of holomorphic sections which are also holomorphic for the
Gauss-Manin connection, so $E$ is a holomorphic subbundle of $H$.

Looking at the $(1,0)$-part of (4.1), we see that the term $[\K_t.u]$
is 'orthogonal' to the fibers of $E$ for the quadratic form $\langle ,
\rangle$ on $H$, simply for reasons of bidegree. By the Griffiths
formula for the curvature of a holomorphic subbundle (see below) this
means that
\be
\langle \Theta^{GM} u, u\rangle=\langle \Theta^E u, u\rangle +\langle
\K_t.u, \K_t.u\rangle
\ee
for sections of $E$. 
 In the standard situation, when the quadratic form is positive
 definite, this formula implies that the curvature of the holomorphic
 subbundle is smaller than the curvature of $H$. In our  case, the
 form is not positive definite, which changes things
 completely.

Note first that $\K_t.u$ is a primitive class. This follows from Lemmas
2.2 and 2.1; by Lemma 2.2 $-\K_t.u$ is cohomologous to $\eta$ which is
primitive by (the proof of) Lemma 2.1.

This implies that the second term on the right hand
 side of (3.4) is equal to the negative of the norm of the cohomology
 class with respect to any Kähler metric we choose on the fiber. (This
 quantity is independent of the 
 choice of Kähler metric.)  Since
 we also have 
 that $\Theta^{GM}=0$, (3.4) implies
$$
\langle \Theta^E u, u\rangle= \|\K_t.u\|^2,
$$
so we get another proof of Theorem 1.1

{\bf Remark:} Formula (3.4) can be obtained as follows. Choose a
holomorphic section of $E$, such that $D'u=0$ at a given point. Then,
at that point,
\be
\Delta\langle u,u\rangle =-\langle \Theta^E u, u\rangle.
\ee
On the other, we can also compute the Laplacian using the Gauss-Manin
connection:
$$
\Delta\langle u,u\rangle =-\langle \Theta^{GM} u, u\rangle +\langle \D'
u, \D'u\rangle =\langle
\K_t.u, \K_t.u\rangle,
$$
by (3.3). Comparing these two formulas we get (3.4).\qed

\section{The twisted case}

In this section we consider a semiample line bundle $\L$ over $\X$,
equipped with a semipositive metric $\phi$ which is assumed to be
strictly positive when restricted to any fiber. We can then take our
underlying Kähler metric to be
$$
i\ddbar (\phi +|t|^2)=:\omega,
$$
so in particular $\omega$ restricts to $i\ddbar\phi$ on any fiber.  
As in the introduction, we choose local coordinates $(t,z)$ on $\X$
that respect the fibration, a local representative $\varphi$ of the
metric $\phi$ with respect to some trivialization of $\L$,  and define
a local vector field $W_\varphi$ as 
the complex gradient of the local function $\dof$. We then let
$$
V_\phi:= \frac{\partial}{\partial t} - W_\varphi.
$$
By a theorem of Schumacher, \cite{Schumacher}, the field $V_\phi$ is
globally defined on $\X$. This is also a consequence of the following
lemma, that we will use repeatedly.
\begin{lma}
\be
\delta_{V_\phi}\ddbar\phi=c(\phi)d\bar t.
\ee
(The function $c(\phi)$ is defined in (1.2) and (1.3).)
\end{lma}
\begin{proof}
Recall that 
$$
\dof=\frac{\partial\varphi}{\partial t}
$$
and that the local vector field $W_\varphi$ was defined by
$$
\delta_{W_\varphi}\omega=\dbar\dof
$$
i e
$$
\delta_{W_\varphi}(\ddbar\varphi)_z=\dbar\dof
$$
on fibers. We have
$$
\ddbar\varphi=(\ddbar\varphi)_z +dt\wedge\dbar_z\dof
+\partial_z\bar{\dof}\wedge d\bar t +\frac{\partial^2\varphi}{\partial
  t\partial\bar t} dt\wedge d\bar t.
$$
Contracting with $V_\varphi$ we get
$$
\delta_{V_\varphi}\ddbar\varphi=-\delta_{W_\varphi}(\ddbar\varphi)_z +\dbar\dof
-|\dbar\dof|^2d\bar t +\frac{\partial^2\varphi}{\partial
  t\partial\bar t}d\bar t= (\frac{\partial^2\varphi}{\partial
  t\partial\bar t}-|\dbar\dof|^2)d\bar t.
$$
By formula (1.2) this equals $c(\varphi)d\bar t$.

\end{proof} 
The lemma
shows in particular that, when $\phi$ is relatively positive,
$V_\phi$ is globally well defined. This is 
so, because either $c(\phi)$ is nonzero, and then (4.1) determines
$V_\phi$ directly, or $i\ddbar\phi$ has one zero eigenvalue. Then
(4.1) shows that $V_\phi$ lies in the eigenspace corresponding to the
eigenvalue zero. This determines $V_\phi$ up to a multiplicative
constant, which must be the same for all local definitions, since $dp$
maps $V_\phi$ to 
$\partial/\partial t$. 

Let $u=u_t$ be a holomorphic section of $E$. Then $u\wedge dt$ is a
holomorphic $(n+1,0)$-form on $\X$, with values in $\L$. We now choose
a representative of $u$ by,
\be
\u:= \delta_{V_\phi}(dt\wedge u).
\ee
This choice of $\u$ is the main point of the argument. If with respect
to our local coordinates
$$
u=u^0=\hat u^0 dz,
$$
where $\hat u^0$ is a function, we get
$$
\u=u^0+dt\wedge v,
$$
with 
$$
v= \delta_{W_\phi} u^0
$$
on fibers. 

Hence $\eta$, defined by $\dbar\u=dt\wedge\eta$ is given
by
$$
\eta=-\dbar v =-k^\phi_t.u^0
$$
on fibers. 
Notice that since on fibers
$$
0=\delta_{W_\phi}(i\ddbar\phi\wedge u^0)=i\dbar\dof\wedge
u^0+i\ddbar\phi\wedge v,
$$
we get
$$
\ddbar\phi\wedge v= -\dbar\dof\wedge u^0.
$$
It follows that
$$
\eta\wedge\omega=\dbar^2\dof\wedge u^0=0,
$$
so $\eta$ is primitive on fibers. Since $\eta$ is primitive we get
for the fiberwise Hodge *-operator that
 $*\eta=-\eta$, so 
since $\eta$ is always $\dbar $-closed on fibers, $\dbar*\eta=0$ too.

\bigskip

\noindent Recall that $\partial^\phi\u=dt\wedge\mu$.
Fix a point in $\Delta$ that we take to be 0. Assume the section
satisfies $D'u=0$ at 0. Since $D' u=P(\mu) dt$ this means that $\mu$
is orthogonal to the space of holomorphic forms. By the curvature formula (2.1)
\be
\langle \Theta^E u_0,u_0\rangle_0= \|\eta\|^2_0-\|\mu\|^2_0 +
p_*(c_n i\ddbar\phi\wedge\u\wedge\bar\u)/dV_t 
\ee
since $\eta$ is primitive and $D'u=0$. Here we have an apparently
negative contribution from the norm of $\mu$, but we know from the
general curvature formula in section 2.1 that it must be possible to
aborb it in the norm squared of $\eta$.
\begin{lma}
\be
\ddbar\phi\wedge\u=c(\phi)\u\wedge dt\wedge d\bar t.
\ee
\end{lma}
\begin{proof} For degree reasons $\ddbar\phi\wedge dt\wedge \u=0$ so
$$
0=\delta_V(\ddbar\phi\wedge dt\wedge \u)=c(\phi) d\bar t\wedge
dt\wedge\u+\ddbar\phi\wedge \u.
$$
\end{proof}
From this, the next lemma  follows immediately.

\begin{lma} At $t=0$
\be
p_*(c_n i\ddbar\phi\wedge\u\wedge\bar\u)/dV_t
=\int_{X_0}c(\phi)|u|^2e^{-\phi}.
\ee
\end{lma}

\begin{lma}On fibers
\be
\dbar\mu=-\partial_z^\phi\eta.
\ee
\end{lma}
\begin{proof}
By definition
$$
\dbar\u= dt\wedge\eta.
$$
Hence
$$
\partial^\phi\dbar\u=- dt\wedge\partial^\phi\eta.
$$
But the left hand side here is
$$
-\dbar\partial^\phi\u-\ddbar\phi\wedge\u,
$$
and $\ddbar\phi\wedge\u$ vanishes on fibers by Lemma 5.2. Hence
$$
\dbar (dt\wedge \mu)= dt\wedge\partial^\phi\eta,
$$
which proves the lemma.

\end{proof}

Since we have also seen that $\mu$ is orthogonal to holomorphic forms,
$\mu$ is the $L^2$-minimal
solution to
$$
\dbar\mu=-\partial^\phi\eta=-\nabla'\eta
$$
on $X_0$. Hence 
$$
\mu=-\dbar^*(\Box'')^{-1}\nabla'\eta.
$$
Here $\nabla'$ is the $(1,0)$-part of the Chern connection for $\L$
restricted to $X_0$, 
$$
\Box''= \dbar\dbar^*+\dbar^*\dbar,
$$
and we will also use 
$$
\Box'=\nabla'(\nabla')^*+(\nabla')^*\nabla'.
$$
Then, on $(n+1)$-forms $\Box'=\Box''+1$ by the Kodaira-Nakano
formula. Hence

$$
\|\mu\|^2=\langle -\dbar^*(\Box'')^{-1}\nabla'\eta,\mu\rangle =\langle 
(\Box'')^{-1}\nabla'\eta,\nabla'\eta\rangle=
$$

$$
= \langle(\Box'+1)^{-1}\nabla'\eta,\nabla'\eta\rangle=
\langle(\Box'+1)^{-1}\eta,(\nabla')^*\nabla'\eta\rangle,
$$
since $\nabla'$ commutes with $(\Box'+1)^{-1}$. 
But $(\nabla')^*\eta=*\dbar*\eta=0$ so
$$
(\nabla')^*\nabla'\eta=\Box'\eta.
$$
Hence
$$
\|\eta\|^2-\|\mu\|^2=
\|\eta\|^2-\langle(\Box'+1)^{-1}\eta,\Box'\eta\rangle =
\langle(\Box'+1)^{-1}\eta,\eta\rangle.
$$
Inserting this in (4.3), using lemma 4.3, we get Theorem 1.2.

\bigskip{\bf Remark:}  The last part of the proof amounts to a
calculation of
$$
\|\eta\|^2-\|\mu\|^2,
$$
 if $\mu$ is the
$L^2$-minimal solution to
$$
\dbar\mu=\partial^{\phi}\eta,
$$
showing in particular that this quantity is nonnegative.
This is formally similar to a classical $L^2$-estimate for the
Beurling transform of a function in the plane. It is different from our
earlier expression for the curvature in the case of a trivial
fibration, \cite{2Berndtsson}, where we estimated instead the
$L^2$-minimal solution to the equation
$$
\dbar u=\dbar\dof\wedge u.
$$
The two expressions turn out to be identical for a trivial fibration,
but the formula from \cite{2Berndtsson} can not be used here, since
$\dbar\dof$ has no global meaning.\qed

\subsection{Infinitesimal triviality}

\noindent We see from Theorem 1.2 that if $\Theta^E u=0$ for {\it
  some} $u$ in $E^0$, then $c(\phi)=0$ on $X_0$ and $\eta=-k_0^\phi.u=0$
so $k_0^\phi=0$ on $X_0$. This last condition says that $V_\phi$ is
holomorphic along $X_0$. By Lemma 2.5 in \cite{2Berndtsson}, it also
follows from $c(\phi)=0$ that
$$
\frac{\partial}{\partial\bar t}|_{t=0}V_\phi=0,
$$
so $V_\phi$ is holomorphic to first order also in directions tranverse
to the fiber. If moreover $\Theta^E$ is degenerate not only for $t=0$
but for all $t$ in $\Delta$, then $V_\phi$ is holomorphic on $\X$.
To get cleaner statements let us assume that $V_\phi$ is holomorphic
on $\X$ in the sequel.
 We
shall then  see that $V_\phi$ lifts to a holomorphic vector field on
$\L$,  $\hat V_\phi$, the flow of which acts linearily on the fibers of
$\L$ and moreover satisfies
$$
\hat V_\phi |\xi|_\phi^2=0.
$$
 
\begin{lma} Assume $\Theta^E=0$ (or is just degenerate) on $\Delta$. Then,
near any point in $\X$ there is a local trivialization of $\L$ with
respect to
which the metric $\phi$ is represented by a local function $\varphi$
such that $V_\phi(\varphi)=0$.
\end{lma}
\begin{proof}Let $\varphi$ be any local function, representing $\phi$
  in some holomorphic frame. Any other representative near the point
  is obtained by 
  subtracting a pluriharmonic function.
By definition 
$$
V_\phi(\varphi)=\delta_{V_\phi}\partial\varphi.
$$
Taking $\dbar$ we get, since $V_\phi$ is holomorphic,
$$
\dbar V_\phi(\varphi)=-\delta_{V_\phi} \dbar\partial\varphi.
$$
By Lemma 5.1 this equals $c(\phi) d\bar t=0$. Hence
$V_\phi(\varphi):=\gamma$ is holomorphic. Locally, we can write $\gamma=
V_\phi(\Gamma)$, with $\Gamma$ holomorphic, and then it is enough to
replace $\varphi$ by $\varphi-2\Re \Gamma$.
\end{proof}

By the lemma we get a covering of $\X$ by open sets $U_i$ over which
there are local frames $e_i$ of $\L$, such that 
$$
V_\phi\log |e_i|_\phi=0.
$$
Then the transition functions $g_{i j}=e_i/e_j$ satisfy
$$
V_\phi\log |g_{i j}|^2=0
$$
so
$$
V_\phi g_{ i j}=0.
$$
We can now define $\hat V_\phi$, by letting it be horizontal with
respect to these local frames.  

\subsection{The Weil-Petersson metric.}

In this subsection we will rewrite formula (1.4) in the case when the
fibers are Riemann surfaces, $\L$ is equal to the relative canonical
bundle $K_{\X/Y}$  and the metric $\phi$ is Kähler-Einstein
on each fiber. We shall see that Theorem 1.2  in this case, together
with results 
 of Schumacher, \cite{Schumacher} implies  the formula of Wolpert,
\cite{Wolpert}, for the curvature of the Weil-Peterson metric.

Notice that in this case, $k_t^\phi= i^*_t(\dbar V_\phi)$ is a
$(0,1)$-form on the fiber $X_t$, with values in $T^{1,0}(X_t)=
-K_{X_t}$. This means that the pointwise norm squared
$$
|k_t^\phi|^2
$$
is a well defined function on $X_t$. By Schumacher's formula,
\cite{Schumacher} Proposition 2, \cite{2Schumacher} Proposition 3, it
is related
to the function $c(\phi)$ by 
$$
c(\phi)= (1+\Box)^{-1}|k_t^\phi|^2.
$$
(Schumacher proves the same formula in any dimension if
$\phi$ is Kähler-Einstein on fibers.) Schumacher also proves that the
form $k_t^\phi$ is harmonic on $X_t$. Moreover, the holomorphic
section $u$ of $K_{X_t}+\L=2K_{X_t}$ is Hodge dual to another harmonic
$(0,1)$-form with values in $-K_{X_t}$
$$
k_t^u:= \bar u_t e^{-\phi}.
$$
The first term in the right hand side of (1.4) can therefore be rewritten as
\be
\int_{X_t}(1+\Box)^{-1}\left(|k_t^\phi|^2\right)\cdot |k_t^u|^2 e^\phi.
\ee
The second term in the right hand side of (1.4) depends on a
$K_{X_t}$-valued $(0,1)$-form 
$$
\eta= k_t^\phi u_t= k_t^\phi \overline{k^u_t}e^{\phi}.
$$
It is easy to check that if $\xi$ is any such form, $\xi e^{-\phi}$ is a
function. The $\Box'$-Laplacian of $\xi$ and the
$\Box$-Laplacian on functions are related by
$$
\Box (\xi e^{-\phi})= e^{-\phi}\,\Box'\xi.
$$
Hence
$$
(1+\Box')^{-1}\eta= 
e^{\phi}(1+\Box)^{-1}( k_t^\phi\overline{k^u_t}). 
$$
Therefore the second term in (1.4) equals
$$
\int_{X_t}(1+\Box)^{-1}( k_t^\phi\overline{k^u_t})
\cdot ( k_t^u\overline{k^\phi_t})e^{\phi},
$$
so altogether
\be
\langle\Theta^E u,u\rangle =
 \int_{X_t}\left((1+\Box )^{-1}|k_t^\phi|^2\cdot
   |k_t^u|^2 + (1+\Box)^{-1}( k_t^\phi\overline{k^u_t})
\cdot ( k_t^u\overline{k^\phi_t})
    \right) e^\phi
\ee

Let us now compare this to Wolpert's formula. Let $u_1, ...u_N$ be a
basis for $H^0(X_0, 2K_{X_0})$ and consider a fibration over the
polydisk $\Delta^N$. We can then perform the earlier construction with
respect to each of the coordinates $t_i$, and obtain corresponding
$-\K_{X_0}$-values $(0,1)$-forms 
$$
k_{t_i}^\phi=:B_i.
$$
Then put (following \cite{Yau et al})
$$
f_{i \bar j}= B_i \bar B_j
$$
and 
$$
e_{i \bar j}=(1+\Box)^{-1}f_{i \bar j}.
$$
We now assume the fibration is such that the Hodge dual of $B_i$ is
$u_i$, so that the fibration contains all possible infinitesimal
deformations of $X_0$. Choosing $u=u_k$ and $ k^\phi_t =B_i$, (5.8)
becomes
$$
(\Theta^E)_{i \bar i \,k \bar k}=\int_{X_0} (e_{i \bar i} f_{k, \bar k}
+e_{ i \bar k} f_{k \bar i} )e^\phi,
$$
which is equivalent to formula (2.2) from \cite{Yau et al} if we use
that the $L^2$ metric on $E$ (the bundle of quadratic differentials)
is dual to the Weil-Petersson metric, see \cite{Hubbard} p 328.

\def\listing#1#2#3{{\sc #1}:\ {\it #2}, \ #3.}

\end{document}